\documentclass[11pt]{article}
\usepackage{amsmath,amssymb,amsthm,graphicx,fixmath}

\usepackage{setspace}
\usepackage{ifthen}
\usepackage{tikz}

\def\Re{\mathbb R}

\def\epsilon{{\varepsilon}}
\providecommand{\remove}[1]{}

\theoremstyle{plain}
\newtheorem{theorem}{Theorem}[section]
\newtheorem{lemma}[theorem]{Lemma}

\newtheorem{proposition}[theorem]{Proposition}

\newtheorem{problem}[theorem]{Problem}

\theoremstyle{definition}
\newtheorem{definition}[theorem]{Definition}
\theoremstyle{remark}
\newtheorem{remark}[theorem]{Remark}

\newcommand{\calC}{\mathcal{C}}
\newcommand{\D}{\mathcal{D}}

\newcommand{\cardin}[1]{\lvert {#1} \rvert}

\def\marrow{{\marginpar[\hfill$\longrightarrow$]{$\longleftarrow$}}}

\def\shakhar#1{{\sc Shakhar says: }{\marrow\sf #1}}

\begin{document}
	
\title{Extended VC-dimension, and  Radon and Tverberg type theorems for unions of convex sets}
\author{Noga Alon\thanks{Department of Mathematics, 
Princeton University, Princeton, NJ 08544, USA and
Departments of Mathematics and Computer Science, Tel Aviv University,
Tel Aviv 69978, Israel.
Research supported in part by NSF grant DMS-2154082.
\texttt{nalon@math.princeton.edu}} 
\and Shakhar Smorodinsky\thanks{Institute for the Theory of Computing, Faculty of Computer and Information Science,
 Ben-Gurion University of the NEGEV, Be'er Sheva 84105, Israel.
Research partially supported by the Israel Science 
Foundation (grant no.~1065/20).  \texttt{shakhar@bgu.ac.il}}\thanks{A preliminary version of this paper appeared in Proc. of the ACM-SIAM Symposium on Discrete Algorithms (SODA 2026)}}
	
	\date{}
	\maketitle

\begin{abstract}
  We prove a new Radon type theorem for unions of convex sets,
settling an open problem posed by Kalai in the 1970s. We also define 
	and study an extension of the notion of the VC-dimension of a 
	hypergraph and apply it to establish an extension of our 
	Radon type theorem to a Tverberg type theorem 
	for unions of convex sets. 
\end{abstract}
		
	
\section{Introduction} \label{sec:intro}
Radon's theorem states that any set of $d+2$ points in $\Re^d$ can
be partitioned into two subsets whose convex hulls intersect. Formally,
given a set $P = \{x_1, x_2, \ldots, x_{d+2}\} \subseteq \mathbb{R}^d$,
there exists a partition of $P$ into two disjoint subsets $P_1$ and $P_2$
($P= P_1 \cup P_2$) such that  $\text{CH}(P_1) \cap \text{CH}(P_2)
\neq \emptyset$, where here and in what follows $\text{CH}(X)$ denotes the convex hull of a set of points $X$. The bound $d+2$ is tight as one 
can easily see by taking any set of $d+1$ affine-independent points 
in $\Re^d$, that is, the vertices of a $d$-dimensional simplex.

This fundamental theorem is a cornerstone in discrete
geometry, providing insights into the structure of point sets and their
convex combinations. It implies various fundamental theorems in geometry including the classical Helly's Theorem, the center point theorem, and more.
Its implications extend to machine learning, statistical learning and computational geometry,
influencing algorithms for geometric separation, point location, and
convex hull computations. An example of such implication is the existence of small so-called $\epsilon$-nets or $\epsilon$-approximations for range-spaces defined by semi-algebraic sets which are a core notion in those areas.
The combinatorial notion of VC-dimension is, in fact, the analog of Radon's bound for abstract set-systems.

An equivalent formulation of Radon's theorem
states that for any set $P$ of $d+2$ points in $\mathbb{R}^d$, there
exists a partition of the set into two subsets $P_1$ and $P_2$ such that
any convex set containing $P_1$ must intersect any convex set containing
$P_2$. 
Radon's theorem serves as a basis for numerous generalizations
and related results, such as the following classical and beautiful
theorem of Tverberg \cite{Tverberg1966}, which further enriches our
understanding of geometric configurations.

\begin{theorem}\label{Thm:Tverberg}(Tverberg's Theorem\cite{Tverberg1966})
Let $ r \geq 2$ be a fixed integer and $d \geq 1$. Then for any set $P$
of $(r-1)(d+1)+1$ points in $\Re^d$ there exists a partition of $P$
into $r$ pairwise disjoint sets $P = \bigcup_{i=1}^r P_i$ such that
$\bigcap_{i=1}^r \text{CH}(P_i) \neq \emptyset$.
\end{theorem}

Note that Radon's theorem is the special case of 
Tverberg's theorem with $r=2$.

Tverberg's Theorem has far-reaching implications in discrete and computational geometry and beyond.
Combined with the colorful Carath\'eodory's theorem of B{\'a}r{\' a}ny \cite{Barany82} it implies, for example, the so-called first selection Lemma (see, e.g., \cite{MATOUSEK}).

An equivalent formulation of Tverberg's theorem states that for any set
$P$ of $(r-1)(d+1)+1$ points in $\mathbb{R}^d$, there exists a partition
of the set into $r$ pairwise disjoint sets $P = \bigcup_{i=1}^r P_i$ such
that for any family of convex sets $C_1,\ldots,C_r$ with $P_i \subset C_i$
for every $i \in [r]$ we have that $\bigcap_{i=1}^r C_i \neq \emptyset$.

In this paper we settle an 
open problem posed in the 1970s by Kalai 
and reiterated in later surveys (e.g., \cite{BaranyKalai2022,Kalai-personal}) that asks for a generalization of Radon's Theorem (and Tverberg's Theorem) to sets which are not necessarily
convex nor even connected, but are the union of a bounded number of convex sets. 

\begin{definition}
    Let $s \geq 1$ be an integer. A set $C$ in $\Re^d$ 
is said to be $s$-convex if it is the union of $s$ convex sets. 
\end{definition}

\begin{problem}\label{prob:Radon}[\cite{BaranyKalai2022}-problem 6.6]
    Determine or estimate 
	the least integer $f=f(d,s,t)$ such that for any set $P$
    of $f$ points in $\Re^d$ there is a partition $P=A \cup B$ such that
    any $s$-convex set containing $A$ must intersect any $t$-convex set
    containing $B$.
\end{problem}

B\'ar\'any and Kalai explain in \cite{BaranyKalai2022} why 
$f(d,s,t)$ is finite using a Ramsey type argument. They mention that 
as is often the case when using hypergraph Ramsey Theory
the upper bound obtained is horrible, and add that it
would be nice to understand the correct behavior of this function.
Despite substantial attention, this problem remained unsolved 
for more than fifty years.

A more general problem is the following:
\begin{problem}\label{prob:Tverberg}
    Determine or estimate
	the least integer $f=f_r(d,s_1,\ldots,s_r)$ such that for
    any set $P$ of $f$ points in $\Re^d$ there is a partition into $r$
    pairwise disjoint sets $P = \bigcup_{i=1}^r P_i$ such that for
    any family of sets $C_1,\ldots,C_r$ with $P_i \subset C_i$ where
    $C_i$ is an $s_i$-convex set for every $i \in [r]$ we have that
    $\bigcap_{i=1}^r C_i \neq \emptyset$.
\end{problem}

Notice that Radon's theorem is equivalent to 
$f(d,1,1)=d+2$ and more generally
Tverberg's theorem is equivalent to $f_r(d,1,\ldots,1) = (r-1)(d+1)+1$.

\subsection{The main results}
 
Our first main
result is a fairly simple proof of the following near-optimal upper bound:

\begin{theorem}\label{thm:main}
$f(d,s,t)=O(dst \log( st+1) )$
\end{theorem}

 Our second main result is an extension of the theorem above to the
 following version of Tverberg's theorem for unions of convex sets and
 any $r \geq 2$:
 \begin{theorem}\label{thm:main-r-partition}
     $f_r(d,s_1,\ldots,s_r)=O \left(d r^2 \cdot \log r \cdot \Pi_{i=1}^r s_i 
\cdot {\ln (1+\Pi_{i=1}^r s_i})  \right)$
 
 \end{theorem}
 
\begin{remark}
	Following our paper,
Chen et al. obtain in \cite{ChenETAL2025} nearly matching lower bounds
for $d \geq 2r-2$. For this range of $d$ and $r \geq 2$ they prove
that for all $s$,
	$$f_r(d,s,\ldots,s) \geq (d-2r+4) s^r.$$
\end{remark}
When each of the $s_i$-convex sets in the last theorem is a union of convex components with bounded overlap (specifically, $(\ell+1)$-wise disjoint), we obtain a sharper bound; see Theorem~\ref{thm:l+1-wise}.

A notable feature of our proofs is that they are mainly 
combinatorial and therefore work for abstract
separable convexity spaces with bounded Radon numbers, as we proceed to explain.

\begin{definition}    
An \emph{abstract convexity space} is a pair $(X, \mathcal{C})$ where $X$ is a set and $\mathcal{C} \subseteq 2^X$ is a family of subsets of $X$ called \emph{convex sets}, satisfying the following axioms:
\begin{enumerate}
  \item $X \in \mathcal{C}$ and $\emptyset \in \mathcal{C}$.
  \item $\mathcal{C}$ is closed under intersections, i.e., if $\D \subseteq \mathcal{C}$, then $\bigcap \D  \in \mathcal{C}$.
\end{enumerate}
\end{definition}
This generalizes the notion of convex sets in $\Re^d$. Abstract convexity spaces arise naturally in combinatorics, geometry, and lattice theory, and provide a unifying framework to study convexity. For more on abstract convexity spaces, see, e.g., \cite{Holmsen2025,MoranYehud20,vandevel1993}.

For a subset $P \subset X$ in an abstract convexity space $(X,\calC)$ define its convex hull $\text{CH}(P)$ to be $\text{CH}(P) = \bigcap_{P \subset S, S \in \calC} S$.

For a set $P \subset X$ we say that $P = P_1 \cup P_2$ is a Radon partition if $\text{CH}(P_1) \cap \text{CH}(P_2)\neq \emptyset$.
We say that an $r$ partition $P=\bigcup_{i=1}^r P_i$ is an $r$-Tverberg partition of $P$ if $\bigcap_{i=1}^r \text{CH}(P_i) \neq \emptyset$.
The {\em Radon number} $r(X, \calC)$ of a space $(X,\calC)$ is the minimum integer $n$ such that any subset $P \subset X$ with $\cardin{P}\geq n$ admits a Radon partition. If $n$ is unbounded then the Radon number is $\infty$. Similarly, the $r$'th Tverberg number $T_r(X,\calC)$ is the minimum integer $n$ such that any subset $P \subset X$ with $\cardin{P}\geq n$ admits an $r$-Tverberg partition. 

A convex set $H \in \calC$ in an abstract convexity space $(X,\calC)$ is called a {\em halfspace} if its complement is also a convex set i.e.,  $X\setminus H \in \calC$.

We say that an abstract convexity space $(X,\calC)$ is {\em separable} if for any two disjoint sets there exists a separating halfspace. Formally, for any $C_1 \in \calC $ and $C_2 \in \calC$  if $C_1 \cap C_2 = \emptyset $ then there exists a halfspace $H$ such that $C_1 \subset H, C_2 \subset  X \setminus H$. 

Our proof technique of Theorem~\ref{thm:main} and Theorem~\ref{thm:main-r-partition} works for abstract convexity spaces which are separable and have a bounded Radon number. Let $(X,\calC)$ be a separable abstract convexity space with Radon number $d=r(X,\calC)$.
Call a subset $A \subset X$ $s$-convex if it is the union of $s$ convex sets in $\calC$.
Let $F=F(d,s,t)$ be the minimum integer so that any set $P$ of cardinality at least $F$ in a separable abstract convexity space $(X, \calC)$ with Radon number $d$ admits
a partition $P = P_1 \cup P_2$ such that any $s$-convex set containing $P_1$ must intersect any $t$-convex set containing $P_2$. We have:

\begin{theorem}\label{thm:main-abstractspaces}
     $F(d,s,t)=O(d st \log {(st+1)})$.
 \end{theorem}

More generally let $F=F_r(d,s_1,\ldots,s_r)$ be the minimum integer so that any set $P$ of cardinality at least $F$ in a separable abstract convexity space $(X, \calC)$ with Radon number $d$ admits
a partition $P = \bigcup_{i=1}^r P_i$
into $r$ pairwise disjoint sets $P = \bigcup_{i=1}^r P_i$ such that for
    any family of sets $C_1,\ldots,C_r$ with $P_i \subset C_i$ where
    $C_i$ is an $s_i$-convex set for every $i \in [r]$ we have that
    $\bigcap_{i=1}^r C_i \neq \emptyset$.

\begin{theorem}\label{thm:main-abstractspaces-tverberg}
     $F_r(d,s_1,\ldots,s_r)=O\left (d r^2 \cdot \log r \cdot \Pi_{i=1}^r s_i 
\cdot {\ln (1+\Pi_{i=1}^r s_i}) \right )$.
 \end{theorem}
 
Theorem~\ref{thm:main-abstractspaces-tverberg} for the special case $s_i=1$ for all $i$ provides an upper bound of $O(d r^2 \log r)$ on the so-called Tverberg number of a separable abstract convexity space in terms of its Radon number $d$. This improves the upper bound of $c(d)r^2 \log ^2 r$ by Bukh~\cite{Bukh2010} who proved it for the more general setting of (not necessarily separable) abstract convexity spaces. 
In~\cite{palvolgyi-radon} P\'alv\"olgyi 
provided an upper bound of the form $O(d^{d^{d^{\log d}}}r)$ which is linear in $r$ but super exponential in $d$.

We describe here the proofs of Theorem \ref{thm:main} and Theorem
\ref{thm:main-r-partition}. Essentially the same proofs establish the corresponding results, Theorem \ref{thm:main-abstractspaces} and Theorem
\ref{thm:main-abstractspaces-tverberg},
for  abstract separable convexity spaces.


 \subsection{Structure}
The rest of this paper is organized as follows. In Section~\ref{sec:prelim} we provide background on VC-dimension and shatter functions, and introduce the $r$-VC-dimension. In Section~\ref{sec:proof-thm-main} 
we describe the proof of Theorem \ref{thm:main}. The proof of Theorem \ref{thm:main-r-partition} is given in Section~\ref{sec:Generalized Tverberg}. After a discussion of 
lower and upper bounds for the relevant functions for some special
cases in Section~\ref{sec:special-cases} we suggest a few open problems in the final Section~\ref{sec:conclusion}.

\section{Preliminaries: VC-dimension and $r$-VC-dimension}
\label{sec:prelim}
\noindent
The Vapnik-Chervonenkis dimension of a hypergraph (which is defined below) is a measure of
its complexity. This notion plays a central role in statistical learning,
computational geometry, and other areas of computer science and
combinatorics (see, e.g.,~\cite{AHW87,ABKKW2006,AMY17,BEHW89,MV18}).  
Many graphs and
hypergraphs that arise in geometry have bounded VC-dimension.
Our proof of Theorem \ref{thm:main} is based on some simple properties of the VC-dimension and the shatter function of the relevant hypergraphs (also defined below). In order to prove Theorem \ref{thm:main-r-partition} we define an extended version of the VC-dimension and show how to bound it in terms of the usual dimension.
\vspace{0.2cm}

 \begin{definition}[VC-dimension] \label{def:VC-dim} The
 \emph{Vapnik-Chervonenkis dimension} $VC(H)$ of a hypergraph $H = (V, E)$
 is the largest integer $d$ such that there exists a subset $S \subseteq
 V$ (not necessarily in $E$) with $|S| = d$ that is \emph{shattered} by
 $E$. A subset $S$ is said to be shattered by $E$ if, for every subset
 $T \subseteq S$, there exists a hyperedge $e \in E$ such that $e \cap
 S = T$.
\end{definition}

\begin{remark}
    Note that in a hypergraph $H=(V,E)$ a subset $S$ is shattered iff 
for every partition $S = A \cup B$ (with $A \cap B = \emptyset$) 
	there exist two hyperedges $e_1 \in E$ and $e_2 \in E$ such that $A \subset e_1$, $B \subset e_2$ and $e_1\cap e_2 \cap S = \emptyset$. 
	We will need a generalization of this notion to partitions 
	into $r>2$ sets. This generalization is given 
	below in Definition~\ref{defgen}.
\end{remark}

Note that by the fact that every two disjoint convex sets in $\Re^d$ are separable by  a halfspace, Radon's Theorem implies that no set of $d+2$ points in $\Re^d$ is shattered by halfspaces. Namely, if ${\cal H}_d$ is the family of all halfspaces in $\Re^d$ then the VC-dimension of the hypergraph $(\Re^d,{\cal H}_d)$ is at most $d+1$.
This simple observation holds for every separable abstract convexity space $(X,\calC)$ with Radon number $d$. In that case the VC-dimension of the hypergraph $(X,\cal H)$ where $\cal H$ is the family of all halfspaces in $\calC$ is at most $d-1$.

 \begin{definition}

The \emph{primal shatter function} of a hypergraph $H = (V, E)$ is
the following function $\pi_H : \mathbb{N} \rightarrow \mathbb{N}$:
\[
\pi_H(m) = \max_{S \subseteq V, |S| = m} |\{ S \cap e : e \in E \}|.
\]
 The value $\pi_H(m)$ represents the maximum number of distinct subsets
 of a set $S$ of cardinality $m$ that can be realized as intersections
 with hyperedges in $E$.
 \end{definition}

The following lemma, known as the Sauer-Shelah-Perles  lemma, provides
an upper bound on the shatter function for hypergraphs with bounded
VC-dimension (See, e.g., \cite{MATOUSEK}):
\begin{lemma}[{\bf Sauer-Shelah-Perles}]
    \label{Lem:Perles-Sauer-Shelah}
    Let $H = (V, E)$ be a hypergraph with 
VC dimension $d$. Then
\[
\pi_H(m) \leq \sum_{i=0}^{d} \binom{m}{i}.
\]
In particular, if $m > d$, then $\pi_H(m) \leq (\frac{em}{d})^d$.
\end{lemma}

In order to tackle Problem~\ref{prob:Tverberg} we need to
develop an analogous notion of a shattered set in hypergraphs with
bounded VC-dimension for partitions with more than $2$ parts. The relevant 
definition follows. The motivation for this definition will become clear from its application in the study of Problem 
\ref{prob:Tverberg}

\begin{definition}[$r$-VC-dimension]
\label{defgen}
    Let $H=(V,E)$ be a fixed hypergraph.
    Fix an arbitrary set $S \subseteq V$ and consider a partition of $S$ into $r$ pairwise disjoint sets $S=S_1 \cup \cdots \cup S_r$. We call such a partition {\em realizable} if there exist hyperedges $e_1,\ldots, e_r \in E$ such that $S_i \subset e_i$ for all $i \in [r]$ and $S \cap \bigcap_{i=1}^r e_i = \emptyset$.    
    A subset $S \subset V$ is said
    to be {\em $r$-shattered} by $E$ if any partition of $S$ into $r$
    pairwise disjoint sets is realizable. For a Hypergraph $H=(V,E)$ and an integer $r \geq 2$ its $r$-VC-dimension is defined to be the supremum of the possible sizes of an $r$-shattered set in $H$.
\end{definition}
\begin{remark}
    For any $2 \leq r_1 < r_2$ if a subset $S$ is $r_1$-shattered then it is also $r_2$ shattered. Indeed, consider any partition $S=A_1\cup A_2 \cup  \cdots \cup A_{r_2}$. Then by uniting $r_2-r_1+1$ sets of that partition we get an $r_1$ partition, say $S=A_1 \cup \cdots \cup A_{r_1-1} \cup B$ for which there is a family of hyperedges $e_1,\ldots,e_{r_1}$ ($A_1 \subseteq e_1, \ldots, A_{r_1-1} \subseteq e_{r_1-1}, B \subseteq e_{r_1}$ for which $S \cap \bigcap_{i=1}^{r_1} e_i = \emptyset$ ) witnessing the partition which is realizable. The same family of hyperedges with $e_{r_1}$ taken $r_2-r_1 +1$ times is a witness for the $r_2$ partition. Thus, the $r_1$-VC-dimension is at most the $r_2$-VC-dimension. 
\end{remark}

The following lemma provides an upper bound on the $r$-VC-dimension as a function of its classical VC-dimension: 

\begin{lemma}
\label{lem:generalized-shatter}
There exists an absolute constant $C$ such that for every
integer $d$ and any hypergraph $H=(V,E)$ 
with VC-dimension $d$ the following holds. For
every integer $r \geq 2$, the $r$-VC-dimension is at most $Cdr^2 \log r$. This bound is nearly optimal: for every $d$ and $r$ 
there is a hypergraph with VC-dimension $d$ that admits an $r$-shattered set of size $\Omega(dr^2)$.
\end{lemma}
The proof of the upper bound in the lemma is described in Section~\ref{sec:Generalized Tverberg}.
The proof of the lower bound is given implicitly in Theorem~\ref{t42}
as explained in the remark following it.

Even though we don't use the following lemma in this paper we feel that it is worth mentioning as it provides an extension of the Sauer-Shelah-Perles Lemma for hypergraphs with $r$-VC-dimension $t$. 
The proof follows immediately from Corollary 1 in \cite{Alon1983}.

\begin{lemma}[$r$-shatter-function]
\label{lem:generalized-Sauer-Shelah}
Let $H=(V,E)$ be a hypergraph with $r$-VC-dimension $t$.
For a subset $S \subset V$ let $\pi_r(S)$ denote the number of realizable $r$-partitions of $S$.
Let $\pi_r(m)$ denote the maximum of $\{\pi_r(S) \mid |S|=m\}$.
Then we have:
$$\pi_r(m)\leq  \sum_{i=0}^t {m \choose i}(r-1)^{m-i} 
	\leq (r-1)^m\sum_{i=0}^t {m \choose i }.$$
\end{lemma}

\section{Proof of Theorem~\ref{thm:main}}
\label{sec:proof-thm-main}
Before proceeding to the proof of Theorem~\ref{thm:main} we need the
following two simple lemmas.
 \begin{lemma}\label{lem:separation}
     Let $C_1$ be an $s$-convex set in $\Re^d$ and $C_2$ a $t$-convex
     set. Assume that $C_1 \cap C_2 = \emptyset$. Then there exist
     $s$ convex polytopes $K_1,\ldots, K_s$ each having at most
$t$ facets whose union covers $C_1$ so that the complement of
     the union covers $C_2$. Namely,  $C_1 \subset \bigcup_{i=1}^s K_i$
     and $C_2 \subset \overline{\bigcup_{i=1}^s K_i}$.
 \end{lemma}

 \begin{proof}
     Since $C_1$ is $s$-convex it can be written as $C_1 = \bigcup_{i=1}^s
     X_i$  for some convex sets $X_1,\ldots,X_s$.  Similarly  $C_2
     = \bigcup_{j=1}^t Y_i$ for $t$ convex sets $Y_1,\ldots,Y_t$.
     Since $C_1 \cap C_2= \emptyset$, for every $i \in [s], j \in [t]$
     we have that $X_i \cap Y_j = \emptyset$ and therefore there exists
     a hyperplane $h_{i,j}$ strictly separating $X_i$ and $Y_j$. Assume
     without loss of generality that the positive open halfspace
     $h_{i,j}^+$ bounded by $h_{i,j}$ contains $X_i$ and the negative open
     halfspace $h_{i,j}^-$ contains $Y_j$. For every $i \in [s]$ let $K_i$
     be the convex polytope which is the intersection $\bigcap_{j=1}^t
     h_{i,j}^+$. Note that $K_i$ is a convex polytope with at most $t$
     facets containing $X_i$ and its complement $\overline{K_i}$ contains
     $C_2$. Thus the union of the polytopes $\bigcup_{i=1}^s K_i$ contains
     $C_1$ and its complement $\overline{\bigcup_{i=1}^s K_i}$ contains
     $C_2$. 
	 This completes the proof of the lemma.
 \end{proof}

 \begin{lemma}\label{lem:boolean-VC-dim}
    Let $l=st$ where $s,t \geq 1$ are integers,
	 and let $H=(P,E)$ be a hypergraph where
    $P$ is a set of points in $\Re^d$ and $S \in E$ is a hyperedge
    if and only if there exists a set $K_1,\ldots,K_s$ of $s$ convex
    polytopes each having at most
$t$ facets such that $S= P
    \cap (\bigcup_{j=1}^iK_j)$. Equivalently, $S$ can be cutoff from $P$ by
    intersecting it with a set consisting of a union of $s$ convex polytopes
    each containing at most $t$ facets. Then the VC-dimension of $H$ is bounded
    by $O(d l \log (l+1))$.
 \end{lemma}
 \begin{proof}
     The proof is standard, hence we only sketch the argument.
	Suppose $X$ is a set of points. By Lemma
	 \ref{Lem:Perles-Sauer-Shelah} we have an upper bound for the
	 number of intersections of $X$ with a halfspace. This number
	 raised to the power $t$ is thus an upper bound for the number 
	 of distinct intersections of $X$ with a polytope having at most
	 $t$ facets. Raising it again to the power $s$ we get an upper 
	 bound for the number of intersections of $X$ with a union of
	 $s$ such polytopes. If $X$ is shattered then the last number 
	 should be at least $2^{|X|}$, providing the required bound.
	 We omit the detailed computation.
	 See also, e.g., \cite{MATOUSEK}(Proposition 10.3.3) for this
	 reasoning.
 \end{proof}

{\noindent \bf Proof of Theorem~\ref{thm:main}:}
Put $l = st$. Let  $H=(\Re^d,E)$ be the hypergraph as in
Lemma~\ref{lem:boolean-VC-dim}. Let $n=O(d l \log (l+1))= O(d st \log (st+1))$
be its VC-dimension. We claim that $f(d,s,t) \leq n+1$. Indeed, Let $P$
be a set of $n+1$ points in $\Re^d$. Since $P$ cannot be shattered by
the hyperedges in $H$ there exists a non-trivial subset $A \subset P$
such that no hyperedge $S \in E$ has the property that $S\cap P= A$. In
other words, there does not exists a set $K$ which is the union of $s$ 
convex
polytopes, each having at most $t$ facets such that $K\cap P = A$. We
claim that the partition $P=A \cup (P\setminus A)$ has the property that
every $s$-convex set containing $A$ must intersect any $t$-convex set
containing  $P\setminus A$. Indeed, assume to the contrary that there
exists an $s$-convex set $C_1$ containing $A$ and a $t$-convex set $C_2$
containing $P\setminus A$ such that $C_1 \cap C_2 = \emptyset$. Then by
Lemma~\ref{lem:separation} there exists a set $K$ which is the union of
$s$ convex polytopes each having at most $t$ facets which contains $C_1$
with a complement $\overline{K}$ that contains $C_2$. In particular, $K$
contains $A$ and $\overline{K}$ contains $P \setminus A$ so $K \cap P =
A$, a contradiction. This completes the proof. \hfill $\Box$

\section{A Generalized Tverberg Theorem}
\label{sec:Generalized Tverberg}
In this section we tackle Problem~\ref{prob:Tverberg}.  In what follows we
provide a bound on $f_r(d,s_1,\ldots,s_r)$. To simplify the presentation
we assume that $s_1=s_2=\cdots=s_r=s$ and abuse the
notation writing $f_r(d,s)$ for $f_r(d,s,s,\ldots ,s)$. Our
proof technique can be easily modified to make the bound sensitive 
for any $r$
integer parameters $s_1,\ldots,s_r$ in the more general setting.

Our argument is based on the notion defined in  \ref{defgen}, which is an extension of the notion of 
VC-dimension to
partitions with more than $2$ sets.

We need Lemma 
\ref{lem:generalized-shatter}. We proceed with its proof. 

\begin{proof}
Let $r\geq 2$ be an integer. 
Suppose that
\begin{equation}
    \label{e31}
   \left (\sum_{i=0}^d {f \choose i}\right  )^r < \left (\frac{r}{r-1}\right )^f
   \end{equation}

We claim that then
any subset $S$ of $f$ vertices cannot be $r$-shattered. Namely, there
exists a partition $S=\bigcup_{i=1}^r S_i$ so that whenever we have $r$
hyperedges $e_1, e_2,\ldots,e_r \in E$ such that $S_i \subset e_i$ for all
$i \in [r]$ it must hold that $S \cap \bigcap_{i=1}^r e_i \neq \emptyset$.

Note that the inequality (\ref{e31}) holds 
for $f = Cdr^2 \log r$ for some absolute constant $C$ so any $r$-shattered set has size at most $f-1$.

To prove the claim suppose it is
false and there is a set $S$ that violates the condition.
For every (ordered) $r$-tuple  of hyperedges $e_1,e_2,\ldots,e_r$
with no common intersection in $S$, every point of $S$ belongs to at
most $r-1$ of these hyperedges. Therefore, these fixed $r$ hyperedges
can be used to provide at most $(r-1)^f$ partitions into $r$ sets
$S_1,\ldots,S_r$. Indeed, each of the $f$ points of $S$
can be allocated to 
at most $r-1$ of the $r$ parts.
If, for example, a point lies in the hyperedges $e_1$, $e_2$
and so on
up to $e_{r-1}$, then in the partition this point can be in 
either $S_1$ or $S_2$ and so on
up to $S_{r-1}$ but not in $S_r$, and the situation
	is similar in each other case.
By the Sauer-Shelah-Perles Lemma~\ref{Lem:Perles-Sauer-Shelah},  
	the number of ordered
$r$-tuples of intersections of hyperedges with $S$ is at most

$$\left(\sum_{i=0}^d {f \choose i} \right)^r$$

Since we have to cover all $r^f$ ordered 
partitions of $S$ into $r$ parts we get

$$\left (\sum_{i=0}^d {f \choose i}\right)^r  \cdot (r-1)^f \geq r^f$$

This contradicts the assumption (\ref{e31})  and completes the proof. 
\end{proof}

We also need the following simple geometric result showing
that $s$-convex sets with no common point can each be 
enclosed in an $s$-convex 
polytope (i.e., a union of $s$ convex polytopes), such that all those $s$-convex polytopes also have no common
point. Moreover, we provide an upper bound on the number of facets of 
each of these polytopes. 

\begin{lemma}\label{lem:r-separation}
 Let $C_1,C_2,\ldots,C_r$ be $r$ sets in $\Re^d$ where each set is an
 $s$-convex set. Assume that $\bigcap_{i=1}^r C_i= \emptyset$. Then there
 exist $r$ sets $K_1,\ldots, K_r$ where each $K_i$ is the union of $s$
	convex polytopes each  having at most $s^{r-1}$ facets 
	such that $C_i
 \subset K_i$ for all $i \in [r]$ and $\bigcap_{i=1}^r K_i = \emptyset$.
\end{lemma}

\begin{proof}
Since $C_i$ is $s$-convex for any $i \in [r]$, it can be
written as $C_i = \bigcup_{j=1}^s X_{i,j}$  for some convex sets
$X_{i,1},\ldots,X_{i,s}$.  Since $\bigcap_{i=1}^r C_i =\emptyset$,
for every $i \in [r]$ we have that $C_i$ is disjoint from $B_i
= \bigcap_{j \neq i} C_j$. Note that each such $B_i$ is the
union of at most $s^{r-1}$ convex sets since it is an $(r-1)$-fold
intersection of unions of $s$-convex sets. We construct the sets
$K_1,\ldots,K_r$ one by one. First, we replace  $C_1$ by $K_1$
separating it from the union of at most $s^{r-1}$ convex sets
$B_1= \bigcap_{j>1} C_j$. As before, this can be done with $K_1$
which is the union of $s$ convex polytopes, each having
	at most $s^{r-1}$ facets. In particular $K_1$ is also $s$-convex. Also $C_1 \subset K_1$. Moreover,  $K_1
\cap \bigcap_{j>1} C_j=\emptyset$. We then apply the same argument to $C_2$
as the intersection of $K_1,C_2,..,C_r$ is empty and all the sets
are $s$-convex. So we can find a set $K_2$ which is $s$-convex 
	and consists of the union of $s$ convex polytopes each having
	at most $s^{r-1}$ facets and such that $C_2 \subset K_2$ and $K_1 \cap K_2 \cap
\bigcap_{j > 2}C_j = \emptyset$. Continuing in the same manner we conclude that
each $C_i$ can be replaced with  such a $K_i$ so that $C_i
\subset K_i$ for all $i \in [r]$, each $K_i$ is the union of $s$
	convex polytopes, each with at most $s^{r-1}$ facets,
	and $\bigcap_{i=1}^r
K_i = \emptyset$.  This completes the proof of the lemma.
\end{proof}

A refinement of Lemma~\ref{lem:r-separation} under an additional bounded-overlap assumption on the convex components
(namely, $(\ell+1)$-wise disjointness) is given in Section~\ref{sec:special-cases}; see Theorem~\ref{thm:l+1-wise}.

We are now ready to prove the following theorem extending
Theorem~\ref{thm:main} to partitions with $r > 2$
parts, and generalizing
Tverberg's theorem to $s$-convex sets:

\begin{proof}[\bf Proof of Theorem~\ref{thm:main-r-partition} (for $s_1=s_2=\ldots =s_r=s$)]
    Put $l = s^r$. Let  $H=(\Re^d,E)$ be the hypergraph as in
    Lemma~\ref{lem:boolean-VC-dim}. Let $d'=O(d l \log (l+1))$ be its VC-dimension. Let $n$ be the maximum size of
    an $r$-shattered set. Note that by Lemma~\ref{lem:generalized-shatter}
    $n = O(d' r^2 \log r)= O(d r^2 {\log r} \cdot  s^r \cdot {\log (s^r +1)} )$

    We claim that $f_r(d,s) \leq n+1$. Indeed, let $P$ be a set of
    $n+1$ points in $\Re^d$. Since $P$ cannot be $r$-shattered by the
    hyperedges in $H$ there exists a partition $P= \bigcup_{i=1}^r P_i$
    such that whenever we have $r$ hyperedges $e_1,\ldots,e_r  \in E$
    with $P_i \subset e_i$ for each $i \in [r]$ it follows that $P \cap
    \bigcap_{i=1}^r e_i \neq \emptyset$. In other words there do not
    exist sets $K_1,\ldots,K_r$ each of which is the union of $s$ convex
	polytopes, each having at most $s^{r-1}$ facets,
	such that $P_i \subset
    K_i$ for every $i \in [r]$ and $\bigcap_{i=1}^r K_i = \emptyset$. We
    claim that the partition $P=\bigcup_{i=1}^r P_i$ has the property
    that for every family of $r$ sets $C_1,\ldots, C_r$ such that for each
    $i \in [r]$ $P_i \subset C_i$ and each $C_i$ is an $s$-convex set
    it must hold that $\bigcap_{i=1}^r C_i \neq \emptyset$. Indeed,
    assume to the contrary that there exist $C_1,\ldots, C_r$ such
    that for each $i \in [r]$ $P_i \subset C_i$, each $C_i$ is an
    $s$-convex set and $\bigcap_{i=1}^r C_i = \emptyset$. Then
    by Lemma~\ref{lem:r-separation} there exist sets $K_1,\ldots,K_r$
    for which  $C_i \subset K_i, \forall i \in [r]$ and each $K_i$ is the
	union of $s$ convex polytopes each having at most $s^{r-1}$ 
	facets and
    $\bigcap_{i=1}^r K_i = \emptyset$, a contradiction. This completes
    the proof.
\end{proof}
\begin{remark}
    Chen et al.~\cite{ChenETAL2025} showed that in the special case 
	where the $s$-convex sets in the theorem are required to be 
	pairwise disjoint then one can improve the bound in 
	Theorem~\ref{thm:main-r-partition} to $C_{d,r}s^{2d+3}$ 
	where $C_{d,r}$ is a constant depending only on $d$ and $r$. 
	We can further improve this bound to $C_{d,r}s^{d+1} \log s$ 
	and generalize it to the case where the $s$-convex sets are 
	not necessarily pairwise disjoint but only $l+1$-wise 
	disjoint, meaning that any $l+1$ sets of them are disjoint 
	for some fixed $1 \leq l$. The proof is essentially the
	same, replacing
	the bound for the number of facets of each polytope in
	$K_i$ by $O_{l,d,r}(s^{d})$ as described in 
	Theorem~\ref{thm:l+1-wise}.
\end{remark}

\section{Improved bounds in special cases}

\label{sec:special-cases}
As before we denote the function 
$f_r(d,s_1,s_2, \ldots ,s_r)$ in which $s_i=s$ for all $i$ by $f_r(d,s)$.

In this section we discuss improved upper and lower bounds for the 
functions $f$ and $f_r$ in several special cases.

\subsection{The Asymmetric Case $f(d,s,1)$}
We start with the following asymmetric case providing the exact value of $f(2,s,1)$, sharp asymptotic bounds on $f(d,s,1)$ for $d 
\geq 4$ and near optimal upper bound on $f(3,s,1)$:

\begin{theorem}
     $$f(2,s,1)=2s+2$$ 
     \newline and for every $d \geq 4$ $$f(d,s,1)= \Theta(d s \log (s+1))$$
    
\end{theorem}
\begin{proof}
    Let $l$ denote the VC-dimension of a hypergraph defined by points in $\Re^d$ with respect to convex polytopes with at most $s$ facets. We first show that $f(d,s,1)= l+1$. Combined with the known bounds on the VC-dimension of such hypergraphs  we get the claimed bounds.
    If $l$ is the VC-dimension of such a hypergraph then there exists a shattered set $P$ of size $l$. Then for any partition $P = P_1 \cup P_2$ there is a polytope $K$ with at most $s$ facets such that $P_2 \subset K$ and $K \cap P_1 = \emptyset$. In particular $P_1$ is contained in the union of the at most $s$ complement halfspaces supporting the facets of $K$, which is an $s$-convex set (consisting of the at most $s$ complement halfspaces), and $P_2$ lies in its complement. 
    Since this holds for any partition it follows that $f(d,s,1) \geq l+1$. In order to show equality we need to show that every $l+1$ point set in $\Re^d$ admits a partition that cannot be realized by intersections with disjoint convex and $s$-convex sets. This follows by the same argument from the fact that any set $P$ of $l+1$ points cannot be shattered in the corresponding hypergraph and therefore there is at least one partition $P=P_1 \cup P_2$ so that no polytope with at most $s$ facets can contain $P_2$ while being disjoint from $P_1$. Hence any $s$-convex set $C$ containing $P_1$ must intersect the convex hull $\text{CH}(P_2)$ for otherwise by separation arguments as above there would be a polytope which is the intersection of $s$ half spaces containing $P_2$ and disjoint from $P_1$, a contradiction.
    It is a simple exercise to see that in the plane ($d=2$), the VC-dimension is $l=2s+1$ and thus $f(2,s,1)=l+1 = 2s+2$. The statement for any fixed $d \geq 4$ follows from the results in
    \cite{CMK2019}.
\end{proof}

\begin{theorem}
    $$f(3,s,1) \leq 4s+1$$
\end{theorem}

\begin{proof}
    We need to prove that for every set $P$ of $4s+1$ points there
    exists a partition $P=A\cup B$ such that any convex set containing
    $A$ must intersect any $s$-convex set containing $B$. Notice that by
    the above arguments it is enough to prove that the VC-dimension of the
    hypergraph $H=(P,E)$ where $E$ is the family of all intersections of
    $P$ with a convex polytope with at most $s$ facets is bounded by $4s$. This fact is proved in \cite{DG1995}. For completeness we include a proof based on an argument in \cite{smoro}.
    Assume to the contrary that there exists a set $P \subset \Re^3$
    of size $4s+1$ that is shattered by $H$. In particular, for any
    partition $P=A \cup B$ if $A$ can be separated from $B$ by a convex
    polytope with at most $s$ facets, then there exists a set of $s$
    half spaces whose union contains $B$ but none of the points in $A$.

    Next, we need the following fact which was proved in \cite{smoro}:
    There exists a $4$-coloring of the points of $P$ such that no
    halfspace that contains at least two points of $P$ is monochromatic.

    Consider such a coloring. By the pigeonhole principle there is a
    monochromatic set $B \subset P$ of size at least $s+1$. We claim that
    $B$ cannot be separated from $A=P\setminus B$ with a union of only
    $s$ halfspaces. Indeed, if such $s$ halfspaces exist then one of
    them must contain at least $2$ points of $B$ and none of the points
    in $A$ so such a halfspace cuts off a monochromatic set of points,
    a contradiction. This completes the proof.

\end{proof}

\subsection{Upper Bounds}
\label{sec:bounded-overlap}
In the special case where each of the sets $C_i$ in Lemma~\ref{lem:r-separation} above has the property that the $s$ convex sets $X_{i,1},\ldots, X_{i,s}$ whose union is $C_i$ are in addition $l+1$-wise disjoint (any 
 intersection of $l+1$ of them is empty) then one can make sure that 
 each polytope defining the sets $K_i$ has at most 
 $O_{r,d,l}(s^{d})$ facets.
 
 \begin{theorem}
 \label{thm:l+1-wise}
     Let $r \geq d+1$ be an integer and $C_1,C_2,\ldots,C_r$ be $r$ sets in $\Re^d$ where each set is an
 $s$-convex set.  Assume that $\bigcap_{i=1}^r C_i= \emptyset$ and that each $C_i$ is the union of s convex sets which are $l+1$-wise disjoint.  Then there
 exist $r$ sets $K_1,\ldots, K_r$ where each $K_i$ is the union of $s$
 convex polytopes each having at most $O_{r,d,l}(s^d)$-facets such 
	 that $C_i
 \subset K_i$ for all $i \in [r]$ and $\bigcap_{i=1}^r K_i = \emptyset$.
 \end{theorem}
 
 \begin{proof}
     
 We follow the notations as in the proof of Lemma~\ref{lem:r-separation}. The main idea is to use the following known fact: If $r-1$ convex 
	 sets $X_{1,j_1},\ldots, X_{r-1,j_{r-1}}$ in $\Re^d$ (for $r-1 \geq d$) have a non-empty intersection then there is a subset of at most $d$ indices $I \subset [r-1]$ such that the 
	 lexicographic minimum point of $\cap _{i \in I} X_{i,j_i}$ equals the lexicographic minimum point of  $\cap_{i=1}^{r-1} X_{i,j_i}$ (see e.g., Lemma~8.1.2 in \cite{MATOUSEK}). Thus, consider 
	 for example the $s$-convex set $K_r$ constructed in 
	 Lemma~\ref{lem:r-separation} which is a union of $s$ convex 
	 polytopes each being the intersection of at most $s^{r-1}$ 
	 halfspaces. We can reduce the number of halfspaces as follows: 
	 The intersection $\cap_{i=1}^{r-1} C_i$ consists of a family of intersections of $r-1$ convex sets $X_{1,j_1},\ldots,X_{r-1,j_{r-1}}$. 
	Each such nonempty intersection can be charged to some 
	 lexicographic minimum point of the intersection of a 
	 subfamily $\{X_{i,j_i}| i \in  I\}$ for some $I \subset [r-1]$ with $|I|=d$. For any such fixed set $I \subset [r-1]$ with $|I| \leq d$, there are at most $l^{(r-1-|I|)}$ ($r-1$)-tuples of 
	 convex sets that charge such a point since this point can 
	 belong to at most $l$ convex sets of the form $X_{i,j_i}$ 
	 for any $i \in [r-1]\setminus I$.  Clearly the intersection 
	 of each such tuple is convex and there are at most 
	 $\sum_{i=0}^d {r-1 \choose i} s^il^{r-1-i} = O_{r,d,l}(s^d)$ 
	 such tuples. Thus, for each of the convex components 
	 $X_{r,i}$ $i \in [s]$ of $C_r$ we can use at most these 
	 many halfspaces. Altogether $K_r$ consists of $s$-convex 
	 polytopes, each having at most  $O_{r,d,l}(s^{d})$ facets. 
\end{proof}

The following simple result shows that the lower bound for
$f_r(d,s)$ proved in  Theorem \ref{t42}
is tight up to a constant factor in
dimension $d=1$.

\begin{theorem}
	\label{t999}
$$
f_r(1,s) \leq r(r-1)(s+1)+1
$$
\end{theorem}
\begin{proof}
Put $n=r(r-1)(s+1)+1$ and let
$0<p_1 <p_2 < \ldots <p_n<1$ be a set of $n$ points on the line.
We have to show that there is a coloring of these points by $r$
colors $0,1,2,\ldots ,r-1$, so that for any $r$ sets
$C_i$, where each $C_i$ is a union of at most $s$ intervals
that covers all points of color $i$,
there is a point that lies in all sets $C_i$.  Naturally, the
coloring we choose colors the points periodically, that is, the
color of $p_i$ is defined to be $i \bmod r$. Let $C_i$
be collections of intervals as above. Note that each $C_i$ must
contain all points of $P$ besides at most 
$(s+1)(r-1)$. Indeed, since it contains all points colored $i$, the
gap between any two consecutive intervals in it contains at most
$r-1$ points, and the same holds for the gap between 
$0$ and its leftmost point and between $1$ and its rightmost point.
(We note that here we can slightly improve the bound since, for
example, the gap between $0$ and the leftmost point of $C_i$
can contain at most $i$ points $p_i$). It follows that if
$n>r(r-1)(s+1)$ then there is a point (of $P$, although that's not
needed) that belongs to all sets $C_i$. This completes the proof.
\end{proof}

\subsection{Lower Bounds}

\begin{remark}
A trivial lower bound for the function $f_r(d,s)$  can be proved by
taking $s$ translated copies of an extremal example for the classical
theorem of Tverberg with pairwise disjoint convex hulls. This gives
the following.
    $$f_r(d,s) > s(d+1)(r-1).$$
    This is, of course, tight for $s=1$, as Tverberg's Theorem  is
tight.
\end{remark}

 The next result shows that for
$s \geq 3$ and any dimension $d \geq 1$,
the lower bound becomes quadratic in $r$. Note that for
$d \geq 2r-2$ a stronger bound of $(d-2r+4)s^r$ is proved in
\cite{ChenETAL2025}, but the bound below holds for all $d$.
As shown in Theorem \ref{t999} above, for $d=1$ the
lower bound is tight up to 
a constant factor for all $r$ and $s \geq 3$. For convenience we
describe the proof for even $r$, a similar bound for odd $r$ follows
from the one for $r-1$.

\begin{proposition}
\label{t42}
For every $d \geq 1, s \geq 3$ and even $r \geq 2$,
$$
f_r(d,s) > \frac{1}{4} \left ( \left \lfloor \frac{d}{2}\right \rfloor+1\right)
\left \lfloor \frac{s-1}{2} \right \rfloor r^2
$$
\end{proposition}

\begin{proof}
Put $m=(\lfloor \frac{d}{2}\rfloor+1)r/2$,
$p=\lfloor \frac{s-1}{2} \rfloor r/2$ and $n=mp$.
Let $P$ be a set of $n$ points on the moment curve
in  $\mathbb{R}^d$ consisting of the $n$ points
$z_i=(t_i,t_i^2, \ldots t_i^d)$, where $0 <t_1<t_2 <  \ldots
<t_n<1$. Let  $I_1,I_2, \ldots ,I_p$ be $p$ open intervals that cover
$(0,1)$ and appear on it in this order, 
where the left endpoint of each interval $I_{j+1}$ is
just slightly smaller than the right endpoint of $I_j$ for each $j$.
The intervals are chosen so that each $t_i$ belongs to exactly one
such interval, and each interval contains exactly $m$ points $t_j$.

In order to complete the proof we show
that for any coloring of the points of $P$  by $r$ colors
there are $s$-convex sets $C_i$, with $C_i$ containing all points
of color $i$, so that the intersection of all the sets $C_i$ is
empty. Each set $C_i$  will be defined as the union of at most
$s$ convex sets, where each of the convex sets will be 
the convex hull of points of color $i$ with first coordinate 
lying in a union of some consecutive intervals $I_q$. The crucial
property of the definition of these convex sets is that for 
each interval $I_q$  there will be an index $i$ so that one of the
convex sets in the definition of $C_i$ will be the set of all
points $z_j=(t_j,t_j^2, \ldots ,t_j^d)$
of color $i$ for which $t_j$ lies in this single interval $I_q$.
Moreover, this will be done in a way that ensures that the number of such
points is always at most $\lfloor d/2 \rfloor$.  

Note that if such
a choice indeed exists, then the intersection of the corresponding 
$r$ sets $C_i$ will be empty as needed. Indeed, otherwise the
intersection is some point $z \in \mathbb{R}^d$ (not necessarily on
the moment curve). Let $t$ denote the first coordinate of this
point $z$. Suppose first that $t$ 
belongs only to one interval $I_q$, and let
$i$ be the color chosen for $I_q$ 
so that there are at most $\lfloor d/2 \rfloor$
points of color $i$ with their first coordinate in  $I_q$. 
Then $z$ has to lie in the convex hull of these points (as any
other convex set among the ones defining 
$C_i$ either has all its points with
first coordinate smaller than $t$ or all its points with 
first coordinate larger than $t$). Since any finite subset of the moment curve is
$\lfloor d/2 \rfloor$-neighborly\footnote{A set of points is $r$-neighborly if any subset of at most $r$ points form a face of the convex-hull of the set.}, this convex hull is disjoint from
the convex hull of all the other points of $P$, implying that
$z$ cannot lie in any other set $C_g$ besides $C_i$. If $t$ belongs
to two intervals, say $I_{q-1}$ and $I_q$  and $i$ is defined as
before, then $z$ cannot lie in $C_i$, since in this case any convex
set among those defining $C_i$
either has all its points with
first coordinate smaller than $t$ or all its points with 
first coordinate larger than $t$. 

It thus only remains to show that we can choose for every interval
$I_q$ a color $i$ with the required properties, making sure that no
fixed color is chosen more than $\lfloor (s-1)/2 \rfloor$ times
(as this way each set $C_i$ will be the union of at most $s$ convex 
sets). To do so we go over the intervals $I_q$ one by one, in an
arbitrary order (for example, from left to right). When dealing
with the interval $I_q$ after handling several previous ones, we
first note that since there are exactly 
$m=(\lfloor d/2 \rfloor+1)r/2$ 
points with first coordinate in $I_q$, there are at
least $r/2$ colors that appears at most 
$\lfloor d/2 \rfloor$ times in $I_q$. As so far we have selected the
colors $i$ for at most $p-1=\lfloor(s-1)/2 \rfloor r/2 -1$
intervals, there are less than $r/2$ colors $i$ that have already
been chosen $\lfloor(s-1)/2 \rfloor $ times. It follows that
there is at least one color $i$ that can be chosen for the interval
$I_q$, implying that the process terminates successfully. This
completes the proof of the theorem.
\end{proof}

\begin{remark}
    Note that for $d=1$, any hypergraph whose vertices are an arbitrary set of points on the line, and whose  hyperedges are the intersections of this set of points with union of $s$ intervals, has VC-dimension at most $2s$. 
    Indeed, it is easily verified that no set of size $2s+1$ is shattered.
    Assume to the contrary that there is a set $P=\{x_1,\ldots,x_{2s+1}\}$ that is shattered where the points are indexed by their increasing order along the line. Notice that every union of $s$-intervals that contains only the $s+1$ points with odd indices $\{x_1,x_3,\ldots, x_{2s+1}\}$ must have one of its intervals containing two consecutive such points and hence also a point with an even index, a contradiction.  Therefore, Theorem~\ref{t42} in dimension $1$ provides a hypergraph with VC-dimension bounded by $D=2s$ and an $r$-shattered set of size 
    at least 
    $$\frac{1}{4} \left \lfloor \frac{s-1}{2} \right \rfloor r^2= \frac{1}{4} \left \lfloor \frac{D-2}{4} \right \rfloor r^2.$$
    This establishes the claimed lower bound in Lemma~\ref{lem:generalized-shatter}
    
\end{remark}

\section{Concluding remarks and open problems}
\label{sec:conclusion}

We established extensions of Radon's Theorem and Tverberg's Theorem for unions of convex sets. The main tools in the proofs are upper bounds for the shatter functions of range spaces with a bounded VC-dimension as well as an extension of these results to $r$-VC-dimension. This extension, defined and studied here, is useful in the study of partitions with more than $2$ parts, which are the ones considered in the classical definition of the VC-dimension.

Our notion of $r$-VC-dimension may be of independent interest beyond the geometric problems studied here. It controls the growth of realizable $r$-partitions (a natural multiclass analogue of the usual shatter function), and thus provides a convenient complexity measure when one needs uniform bounds simultaneously over many parts. We expect that this parameter may be useful in other settings like the study of realizability of $r$-labelings or $r$-way classification rules in machine learning.

As mentioned in the introduction, already for $s_i=1$ for all $i$ the upper bound provided in our Theorem~\ref{thm:main-abstractspaces-tverberg} is $O(d r^2 \log (r+1))$. It improves the upper bound of $c(d)r^2 \log ^2 (r+1)$ by Bukh~\cite{Bukh2010} who proved it for the more general setting of (not necessarily separable) abstract convexity spaces. 
In~\cite{palvolgyi-radon} P\'alv\"olgyi provided an upper bound of the form $O(d^{d^{d^{\log d}}}r)$ which is linear in $r$ but super exponential in $d$.
It will be interesting to decide if one can get rid of the 
separability assumption and prove the bound in Theorem 
\ref{thm:main-abstractspaces-tverberg} for all
abstract convexity spaces.

While our upper bounds and the lower bounds
in \cite{ChenETAL2025}
for the functions $f(d,s,t)$ and $f_r(d,s_1,s_2, \ldots ,s_r)$ 
are not very far from each other, the problem of determining them 
precisely remains open for most values of the parameters. It 
will be interesting to close the gap in these bounds. 
 
\section*{Acknowledgements}
We thank Gil Kalai for fruitful discussions.

\end{document}